%% file: kozlov.tex
\documentclass{amsart}
\input{definitions.tex}
\title[Mech. Systems on Surfaces]{A Note on Integrable Mechanical Systems on Surfaces}
\author{Leo T. Butler}
\email{l.butler@cmich.edu}
\address{120 Pearce Hall, Department of Mathematics, Central Michigan University, Mt. Pleasant, MI, 48859}
\thanks{The author thanks G. Knieper for his helpful comments and suggestions.}

\begin{document}

\begin{abstract}
  Let $\surface[S]$ be a compact, connected surface and $H \in C^2(T^*
  \surface[S])$ a Tonelli Hamiltonian. This note extends
  V.~V. Kozlov's result on the {\ec} of $\surface[S]$ when $H$ is
  real-analytically integrable, using a definition of
  topologically-tame integrability called semisimplicity. Theorem:
  If $H$ is $2$-semisimple, then $\surface[S]$ has non-negative {\ec};
  if $H$ is $1$-semisimple, then $\surface[S]$ has
  positive {\ec}.
\end{abstract}

\maketitle

\section{Introduction}
\label{sec:intro}

\subsection{Kozlov's Theorem} \label{sec:koz} Say that a
\textem{natural} mechanical system is a Hamiltonian that is a sum of
kinetic and potential energies. Let $\surface[S]$ be a compact surface
and $H : T^* \surface[S] \to \R$ be an analytic natural mechanical
system. Kozlov proved if $H$ enjoys a second, independent analytic
integral $F$, then the {\ec} of $\surface[S]$ is non-negative and so
it is homeomorphic to $\sphere{2}$, $\T^2$ or a non-orientable
quotient thereof~\cite{MR556099}. The hypotheses of Kozlov's theorem
can be relaxed as follows: (i) $H$ need only be assumed to be
fibre-wise strictly convex and super-linear (i.e. `Tonelli'); (ii)
analyticity can be reduced to the combined hypotheses that $H$ and $F$
are $C^2$, and that there is an energy level $H^{-1}(c)$ where $c >
\min \set{H(x,0) \mid x \in \surface[S]}$, such that the critical set
of $F$ intersects a fibre of the foot-point projection in only finitely
many points~\cite{MR1411677}.

The present note has two aims. First, it presents a proof of Kozlov's
theorem based on the theory of semisimplicity developed
in~\cite{MR2136534}; see definition \ref{def:gss} below. Second, it
extends Kozlov's theorem to non-commutatively integrable Tonelli
Hamiltonians. The latter is a non-trivial extension:
in~\cite{MR2256655} there are Tonelli Hamiltonians that are
constructed which are non-commutatively integrable and semisimple on
the unit disk bundle, but which are \textem{not} tangent to a
semisimple singular Lagrangian fibration. In essence, the na\"{i}ve
trick of discarding extra integrals to achieve complete integrability
necessarily expands the critical set, and in the above-quoted example,
the critical set expands from a real-analytic set to a wild set
analogous to the Fox wild arc.

\subsection{Non-commutative integrability} \label{ssec:ci} Let
$\Sigma$ be a smooth $n$-dimensional manifold. The canonical Poisson
structure on the cotangent bundle $T^* \Sigma$ permits one to define a
Poisson algebra structure on $\cinfty{T^*\Sigma}$ and consequently
each smooth function $H : T^* \Sigma \to \R$ induces a Hamiltonian
vector field $X_H$ defined by
\begin{equation} \label{eq:xh}
  X_H = \pb{H}{\ } = \sum_{i=1}^n -\didi{H}{x^i} \didi{}{y^i} + \didi{H}{y^i} \didi{}{x^i},
\end{equation}
where $(x^i, y^i)$ are canonical coordinates. A first integral of the
Hamiltonian vector field $X_H$ is a smooth function $F$ which Poisson
commutes with $H$: $\pb{H}{F} = 0$.

For a subspace $\A \subset \cinfty{T^* \Sigma}$ and $p \in T^*
\Sigma$, let $\d_p \A = \spn{ df(p) \mid f \in \A}$. Following
\cite{MR2031454}, the differential dimension of $\A$ is defined to be
$\sup_{p} \dim \d_p \A$. Let $\A \subset \cinfty{T^* \Sigma}$ be a
subspace of first integrals of $H$ that contains $H$ and let
$\centre{\A}$ be the subspace of $\A$ which Poisson commutes with all
of $\A$. Let $k$ (resp. $l$) be the differential dimension of $\A$
(resp. $\centre{\A}$). Say that a point $p \in T^* \Sigma$ is
\textem{regular} for $\A$ if $\dim \d_q \A =k$ and $\dim \d_q Z(\A) =
l$ for all $q$ in the $\A$-level set passing through $p$. We say that
$H$ (or $\A$) is \textem{non-commutatively integrable} if $k+l = 2n$
and the set of regular points is dense.

Nehoro{\v{s}}ev \cite{MR0365629}, who generalized the Liouville-Arnol'd
theorem \cite{MR1345386}, proved that if $H$ is non-commutatively
integrable and $p$ is a regular point, then there is a neighbourhood
$U$ with coordinate chart $(\theta,I,x,y) : U \to \T^k \times \R^k
\times \R^{2(n-k)}$, where the Poisson bracket is canonical and
$H=H(I)$. Dazord and Delzant \cite{MR906389} globalized this result
and showed that the regular point set $X \subset T^* \Sigma$ is fibred
by isotropic $k$-dimensional tori $\fibre{F}=\T^k$ and the quotient of
$X$ by these fibres, $P$, is a Poisson manifold with a foliation
$\zeta$ by symplectic leaves. When this foliation is a fibration, one
has the following diagram
\begin{equation}
  \label{eq:p-fibration}
  \xymatrix{
    \fibre{F} \ar@{^{(}->}[r]^{\iota_{\fibre{F}}} &
    X \ar@{^{(}->}[r]^{\iota_X} \ar@{->>}[d]|(.4){g} \ar@{.>>}[dr]|(.4){G} &
    T^* \Sigma \\
    S \ar@{^{(}->}[r]^{\iota_S} &
    P \ar@{->>}[r]^{\projection{P}} &
    Q
  }
\end{equation}
where $\iota_{\bullet}$ is an inclusion map, $g$ is the fibration
map, $S$ is a symplectic leaf of $P$, $Q$ is an integral affine
manifold of dimension $k$ and $\projection{P}^* \cinfty{Q}$ is the
centre of $\cinfty{P}$ which induces the Hamiltonian vector fields
that are tangent to the isotropic fibres of $g$.

\subsection{Geometric semisimplicity} 
Let us abstract the notion of \completeintegrability. A smooth flow
$\varphi : M \times \R \to M$ is {\em integrable} if there is an open,
dense subset $R \subset M$ that is covered by angle-action charts
which conjugate $\varphi$ to a translation-type flow on the tori of
$\T^k \times \R^l$. There is an open dense subset $L \subset R$ fibred
by $\varphi$-invariant tori; let $f : L \to B$ be the induced smooth
quotient map and let $\Gamma = M - L$ be the {\em singular set}. If
$\Gamma$ is a tamely-embedded polyhedron, then $\varphi$ is said to be
$k$-{\em semisimple} with respect to $(f,L,B)$, or just
semisimple~\cite{MR2136534}. Of most interest is when $\varphi$ is a
Hamiltonian flow on a cotangent bundle or possibly a regular
iso-energy surface.
\begin{definition}[\cf \cite{MR897007,MR2136534}] \label{def:gss}
  A Hamiltonian flow is {\em geometrically $k$-semisimple} if it is
  $k$-semisimple with respect to $(f,L,B)$ and $f$ is an isotropic
  fibration.
\end{definition}
In this case, because the fibres of $f$ are isotropic, $\varphi$ is
non-commutatively integrable, so geometric semisimplicity is a
topologically-tame type of non-commutative
integrability. Taimanov~\cite{MR897007} introduced a related notion of
geometric simplicity, see sections 2.2-2.3 of~\cite{MR2136534} for
further discussion. If $\varphi$ is real-analytically
non-commutatively integrable, then the triangulability of
real-analytic sets implies that $\varphi$ is geometrically semisimple;
and $B$ may be taken to be a disjoint union of open balls. On the
other hand, geometric semisimplicity is a weaker property than
real-analytic non-commutative integrability~\cite{MR2136534}. A basic
question is:

\begin{question} \label{q:a}
  What are the obstructions to the existence of a geometrically
  semisimple (resp. semisimple, completely integrable) flow?
\end{question}

\subsection{Results}
Here are the two main theorems of this note. In both cases,
$\surface[S]$ is a compact, connected surface and $H : T^* \surface[S]
\to \R$ is a $C^2$ Tonelli Hamiltonian.

\begin{thm}[\cf Kozlov~\cite{MR556099}]
  \label{thm:kozA}
  If $H$ is $1$- or $2$-semisimple, then $\surface[S]$
  has non-negative {\ec}.
\end{thm}

\begin{thm}
  \label{thm:kozB}
  If $H$ is $1$-semisimple, then
  $\surface[S]$ is homeomorphic to $\sphere{2}$ or $\R P^2$.
\end{thm}

\begin{remark}
  In two degrees of freedom, as here, $1$- or $2$-semisimplicity
  implies the fibres of the fibration are isotropic, so $1$- or
  $2$-semisimplicity implies \textem{geometric} $1$- or
  $2$-semisimplicity.
\end{remark}

\begin{remark}
  In an earlier version of this note~\cite{arxiv1208.1406}, Theorem
  \ref{thm:kozB} was proven under the additional hypothesis that $H$
  is \textem{reversible}. That proof uses a result of Glasmachers \&
  Knieper~\cite{MR2887674,MR2746954} concerning a zero-entropy
  geodesic flow of a reversible Finsler metric on $\T^2$. The present
  note dispenses with the reversibility hypothesis.
\end{remark}

\begin{remark}
  Bangert has asked a series of questions in \cite{MR2395221}
  concerning integrable Tonelli Hamiltonians which are integrable in a
  weaker sense--the additional integral need only be independent of
  $H$ on an open dense set. These questions are very interesting but
  beyond the scope and techniques of this paper. Likewise, the paper
  by Bialy proves an extension of Kozlov's theorem for geodesic flows
  when the additional integral satisfies his condition $\aleph$
  \cite{MR2671281}. That work relies on properties of minimizing
  geodesics.
\end{remark}

\section{Preliminaries}
\label{sec:prelim}

Let us recall a few items concerning a geometric semisimplicity. Let
$\Sigma$ be a compact smooth manifold and $H : T^* \Sigma \to \R$ be a
$C^2$ Tonelli Hamiltonian that is geometrically semisimple with
respect to $(f,L,B)$. The complement $\Gamma = T^* \Sigma - L$ is a
tamely embedded polyhedron, so the number of components of $L \cap
H^{-1}((-\infty, c])$ is finite for any $c$. \cite[Lemma
15]{MR2136534} implies that there is a component $L_i \subset L$ such
that $\pi \iota_{L_i}$ has a finite-index image in $\pi_1(\Sigma)$:
\begin{equation}
  \label{eq:pi1-epi}
  \xymatrix{
    \pi_1(L_i) \ar[r]^{\iota_{L_i,*}} \ar@/^2em/[rr]^{\pi_* \iota_{L_i,*}} & \pi_1(T^* \Sigma) \ar[r]^{\pi_*} & \pi_1(\Sigma).
  }
\end{equation}

Suppose that $B_0 \subset B$ is a nowhere dense subset such that
$f^{-1}(B_0) \cup \Gamma = \Gamma_1$ is tamely embedded polyhedron
whose complement $L_1 = f^{-1}(B_1)$, $B_1 = B - B_0$, is dense. One
calls $(f_1=f|L_1, L_1, B_1)$ a refinement of $(f,L,B)$. In
\cite[Lemma 18]{MR2136534} it is proven that
\begin{proposition}
  \label{prop:refinement}
  If $\dim B \leq 2$, then $(f,L,B)$ has a refinement $(f_1,L_1,B_1)$
  such that each component of $B_1$ is homotopy equivalent to either a
  point or $\sphere{1}$.
\end{proposition}

\begin{figure}[h]
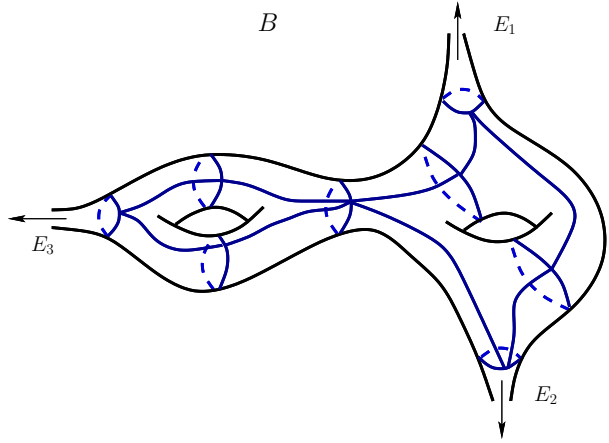

  \centering
  \includepdftex[8cm]{refinement.pdf_t}
  \caption{Schematic proof of ~\ref{prop:refinement}: We cut the base
    $B$ along the blue curves; the ends $E_j$ are cylinders and the
    ``compact'' part of $B_1$ is a union of disks.}
  \label{fig:refinement}
\end{figure}

Let $\zero{\Sigma} \subset T^* \Sigma$ be the zero section of the
cotangent bundle of $\Sigma$. An \textem{exact} Lagrangian graph
$\Lambda \subset T^* \Sigma$ is the graph of an exact $1$-form;
$\zero{\Sigma}$ is an example. A Tonelli Hamiltonian $H : T^* \Sigma
\to \R$ is fibre-wise strictly convex and grows super-linearly in the
fibres. Consequently, there is a $c \in \R$ such that $\set{H \leq c}$
contains $\zero{\Sigma}$ and therefore an exact Lagrangian graph. Let
$\mc$ be the infimum of the set of $c$ such that $\set{H \leq c}$
contains an exact Lagrangian graph; this is {\mane}'s critical
value. For all $c > \mc$, the sublevel set $\set{H \leq c}$ is
fibre-wise strictly convex and contains an exact Lagrangian graph. The
Hamiltonian flow of $H$ on an energy level $c$ above $\mc$ is, up to
an orbit equivalence, the geodesic flow of a Finsler metric.

\section{Proofs}
\label{sec:proofs}
\let\Sigm\Sigma
\renewcommand{\Sigma}[1][S]{\surface[#1]}

Proposition \ref{prop:refinement} allows us to prove Theorem
\ref{thm:kozA}. The manifold $\Sigm$ in the previous section is the
surface $\Sigma$.
\begin{proof}[Theorem \ref{thm:kozA}]
  Suppose that $H$ is geometrically $2$-semisimple. We will deal with
  the case of $1$-semisimplicity below.
  By Proposition \ref{prop:refinement} we can suppose that each
  component of the base of the fibration $f$ is homotopy equivalent to
  a point or a circle. Thus, each component of $L_i \subset L$ is
  homotopy equivalent to $\T^2$ or a $\T^2$-bundle over $\T^1$. In
  both cases, $\pi_1(L_i)$ is solvable, and so $\pi_1(\Sigma)$
  contains a solvable subgroup of finite index. Since $\Sigma$ is a
  surface, the theorem is proved.
\end{proof}

\begin{proof}[Theorem \ref{thm:kozB}]
  To prove Theorem \ref{thm:kozB}, we must adapt the diagram in
  \eqref{eq:p-fibration} to our needs. In this case, the fibre $F =
  \T^1$, the base of the fibration $P (=B)$ is a $3$-dimensional Poisson
  manifold with a foliation $\zeta$ by symplectic surfaces $S$. The
  foliation $\zeta$ is a fibration, in fact, because the Casimirs of $P$
  are functionally dependent on the reduction of the Hamiltonian
  $H|X$. In this case, the quotient of $P$ by $\zeta$, $Q$, is a finite
  union of $1$-manifolds: $Q = \cup_i Q_i$ where $Q_i \simeq \R$ or
  $\T^1$. Since $H|X = h \circ G$ for some $h \in C^2(Q)$, the Tonelli
  property of $H$ implies that no component of $Q$ is a circle.

  It follows that there is a component $X_i = G^{-1}(Q_i)$ such that 
  \begin{equation}
    \label{eq:pi1-epi-X}
    \xymatrix{
      \pi_1(X_i) \ar[r]^{\iota_{X_i,*}} \ar@/^2em/[rr]^{\pi_* \iota_{X_i,*}} & \pi_1(T^* \Sigma) \ar[r]^{\pi_*} & \pi_1(\Sigma)
    }
  \end{equation}
  has a finite index image. Since $X_i$ is homotopy equivalent to an $F
  = \T^1$-principal bundle over the symplectic surface $S_i$, a leaf of
  $\zeta | X_i$, it remains to examine the possibilities.

  \subsection*{$S_i$ is compact} In this case, $G^{-1}(q)$ is compact
  for any $q \in Q_i$, and therefore it must be a connected component of
  an energy level. Since, above the critical value, the energy levels
  are connected, $G^{-1}(q)$ is an energy level. If the {\ec} of
  $\Sigma$ is negative, then $\pi_1(\Sigma)$ contains no non-trivial
  normal abelian subgroups. Therefore, the inclusion $F \hookrightarrow
  T^* \Sigma$ is null-homotopic; this implies that all orbits of the
  Tonelli Hamiltonian in a super-critical energy surface are
  contractible--absurd. Therefore, the {\ec} of $\Sigma$ must be
  non-negative.

  Since every orbit of the Tonelli Hamiltonian is closed on a
  super-critical energy level, the {\ec} of $\Sigma$ must be positive.

  \subsection*{$S_i$ is non-compact} Let $c$ be an energy level such
  that $\pi(S_i) = G(H^{-1}(c)) = q$. Let $X_c = G^{-1}(q)$, $g_c =
  g|X_c$ and $P_c = g(X_c)$. $X_c \subset H^{-1}(c)$ has a complement
  $\Gamma_c = \Gamma \cap H^{-1}(c)$ and is fibred by $F = \T^1$. The
  Hamiltonian flow of $H$ restricted to $H^{-1}(c)$ is therefore
  $1$-semisimple with respect to $(g_c,X_c,P_c)$. Now, $S_c$ is a
  symplectic leaf of the foliation $\zeta$ and therefore is a
  connected symplectic surface. By Proposition
  \ref{prop:refinement}, there is a refinement $(g_c', X_c', P_c')$ such
  that each component of $P_c'$ is homotopy equivalent to a point or
  $\T^1$. Moreover, by \cite[Lemma 15]{MR2136534}, the inclusion of one
  of the components of $X_c'$ in $T^* \Sigma$ is almost surjective on
  $\pi_1$. But the components of $X_c'$ are homotopy equivalent to
  $\T^1$ or $\T^2$ (principal $\T^1$-bundles over $*$ and $\T^1$
  respectively).

  Therefore, $\pi_1(\Sigma)$ contains a finite-index abelian
  subgroup. Hence the {\ec} of $\Sigma$ is non-negative (this completes
  the proof of Theorem~\ref{thm:kozA}). It therefore remains to prove
  that the {\ec} of $\Sigma$ is positive. To do so, we will prove

  \begin{lemma}
    \label{lem:rkh1}
    If $H$ is $1$-semisimple, then $\dim \homology{1}{\Sigma}{\Q} = 0$.
  \end{lemma}

  \begin{proof}[Lemma \ref{lem:rkh1}]
    Let $\varphi^c$ be the Hamiltonian flow of $H$ restricted to the
    iso-energy set $H^{-1}(c)$.

    Let $X' \subset X_c'$ be a component of $X_c'$ and let
    $\xymatrix{F' \ar@{^{(}->}[r]^{} & X' \ar@{->>}[r]^{} & P'}$ be
    the induced fibration of $X'$ by the closed orbits of
    $\varphi^c$. Each orbit is homologous to the homology class of the
    fibre $F'_{p'}$.

    Suppose that $\dim \homology{1}{\Sigma}{\Q} > 0$. 
    Let $\omega \in \homology{1}{\Sigma}{\Q}$ be an integral homology
    class. Each such class $\omega$ contains a closed geodesic, and so
    contains the projection of a closed orbit $\gamma$ of the
    Hamiltonian flow of $H$ restricted to $H^{-1}(c)$. If $\gamma
    \subset X_c'$, then $\omega = \homologyclass{\gamma} \in \Z
    \homologyclass{F'_{p'}}$ for some regular fibre $F'_{p'}$ as in
    the preceding paragraph. If $\gamma \not\subset X_c'$, then, by
    the density of $X_c'$ and continuity in initial conditions of
    $\varphi^c$, each multiple of $\gamma$ is approximated by a broken
    integral curve in $X_c'$, {\ie} a curve $w$ of the form

    \[
    w(t) =
    \begin{cases}
      \varphi^c_{2Tt}(p) & \text{ if } 0 \leq t \leq \frac{1}{2} \\
      \delta_p(t)        & \text{ if } \frac{1}{2} \leq t \leq 1.
    \end{cases}
    \]
    where $\delta_p : [\frac{1}{2},1] \to H^{-1}(c) \cap T^*_{\pi(p)}
    \Sigma$, $p \in X_c'$ and $\varphi^c_T(p) = p$. It follows that
    some multiple of the homology class $\omega$ is homologous to a
    closed orbit of $\varphi^c | X_c'$ and therefore that $\Q \omega
    \subset \Q \homologyclass{F_p}$.

    This proves that $\homology{1}{\Sigma}{\Q}$ is a finite union of
    $1$-dimensional sub-spaces and therefore it is at most
    $1$-dimensional.

    If $\dim \homology{1}{\Sigma}{\Q} = 1$, then $\Sigma$ is a Klein
    bottle and has a double cover $\cover{\Sigma} \to \Sigma$ where
    $\cover{\Sigma}$ is a $2$-torus. The pulled-back Hamiltonian
    $\cover{H}$ is $1$-semisimple. The conclusion of the previous
    paragraph implies that $\dim \homology{1}{\cover{\Sigma}}{\Q} \leq
    1$--an absurdity. Therefore, $\dim \homology{1}{\Sigma}{\Q}$ must
    be $0$.
    
  \end{proof}

  Given Lemma \ref{lem:rkh1}, the proof of Theorem \ref{thm:kozB} is
  complete.

\end{proof}

\bibliographystyle{amsplain}
\bibliography{kozlov}

\end{document}

%% file: definitions.tex
\usepackage{amssymb,graphicx,color}
\usepackage[all,knot]{xy}
\xyoption{import,arc}
\usepackage{hyperref}

\makeatletter
\def\didi#1#2{{\def\@argone{#1}\frac{\partial \ifx\@argone\@empty{\phantom{#2}}\else#1\fi}{\partial #2}}}
\makeatother

\newcommand{\R} [0] { {\bf{R}} }
\newcommand{\T} [0] { {\bf{T}} }
\newcommand{\Z} [0] { {\bf{Z}} }

\newcommand{\Q} [0] { {\bf{Q}} }

\newcommand{\sphere}[1]{\mathbf{S}^{#1}}

\newcommand {\A} [0] { {\mathfrak A} }

\newcommand{\zero}[1]{0(#1)}

\newcommand{\cinfty}[1]{C^{\infty}(#1)}
\newcommand{\set}[1]{\left\{ #1 \right\}}
\newcommand{\textem}[1]{{\em #1}}
\newcommand{\centre}[1]{Z(#1)}
\renewcommand{\d}[0]{{\mathrm d}}

\newcommand{\ec}[0]{Euler characteristic}
\newcommand{\mane}[0]{Ma\~n\'e}
\newcommand{\completeintegrability}{complete and non-commutative integrability}

\newcommand{\mc}{{\mathfrak c}}

\newcommand{\Cw}[1]{{\rm C}^{\omega}{\rm (}[1]{\rm )}}

\newcommand{\ddim}[1]{{\rm ddim(}[1]{\rm )}}
\newcommand{\drank}[1]{{\rm drank(}[1]{\rm )}}

\newcommand{\pb}[2]{{ \left\{ #1, #2  \right\} }}
\newcommand{\spn}[1]{{\rm  span} \left\{ #1  \right\}}

\newcommand{\ie}{{\it i.e.}}
\newcommand{\cf}{{\it c.f.\ }}

\newcommand{\fibre}[1]{{\mathit{#1}}}

\newcommand{\surface}[1][S]{\mathfrak{#1}}
\def\includepdftex[#1]#2{%
  \resizebox{#1}{!}{\input{#2}}}
\newcommand{\projection}[1]{\pi_{#1}}
\newcommand{\homology}[3]{{\mathrm H}_{#1}(#2;#3)}
\newcommand{\homologyclass}[1]{\left[ #1 \right]}
\newcommand{\cover}[1]{\hat{#1}}

\makeatletter
\let\@oldar\ar
\def\ar{\hskip 0.2em \@oldar}
\makeatother

\ifcsname publname\endcsname
\def\omega{\Omega}
\else
\newtheorem{thm}{Theorem}
\newtheorem{lemma}{Lemma}[section]
\newtheorem{proposition}{Proposition}[section]
\newtheorem{definition}{Definition}
\newcounter{remark}
\renewcommand{\theremark}{\Roman{remark}}
\newenvironment{remark}{\begin{trivlist}\item[]\refstepcounter{remark}%
	{\tt Remark\ \theremark.}\ \nobreak\noindent\rm\ignorespaces}{%
	\end{trivlist}}
\fi
\newtheorem{question}{Question}

\let\ltbproof=\proof\def\proof[#1]{\ltbproof[Proof of #1]}
